\documentclass[12pt]{amsart}

\usepackage{enumerate}
\usepackage{amssymb}
\usepackage{amsmath,amscd}
\usepackage[colorinlistoftodos,disable]{todonotes}
\usepackage{amsthm}

\newcommand{\Isom}{\operatorname{Isom}}

\newcommand{\Z}{\mathbb{Z}}

\newcommand{\R}{\mathbb{R}}
\newcommand{\C}{\mathbb{C}}
\newcommand{\HH}{\mathbb{H}}

\newcommand{\OO}{\operatorname{O}}
\newcommand{\SO}{\operatorname{SO}}
\newcommand{\so}{\frak{so}}

\newcommand{\U}{\operatorname{U}}
\newcommand{\SU}{\operatorname{SU}}
\newcommand{\su}{\frak{su}}
\newcommand{\PSU}{\operatorname{PSU}}
\newcommand{\Sp}{\operatorname{Sp}}
\newcommand{\spp}{\frak{sp}}

\newtheorem{theorem}{Theorem}

\newtheorem{lemma}[theorem]{Lemma}
\newtheorem{proposition}[theorem]{Proposition}

\newtheorem{maintheorem}{Theorem}

\theoremstyle{definition}

\theoremstyle{remark}
\newtheorem{remark}[theorem]{Remark}
\newtheorem{example}[theorem]{Example}

\title{Lifting isometries of orbit spaces}
\date{}

\author[R.~A.~E.~Mendes]{Ricardo~A.~E.~Mendes}
\address{University of Oklahoma, USA}
\email{ricardo.mendes@ou.edu}

\thanks{The author has been supported by the NSF grant DMS-2005373}



\begin{document}

\begin{abstract}
Given an orthogonal representation of a compact group, we show that any element of the connected component of the isometry group of the orbit space lifts to an equivariant isometry of the original Euclidean space. Corollaries include a simple formula for this connected component, which applies to ``most'' representations.
\end{abstract}

\maketitle
\todo{added support information at the bottom}

\section{Introduction}
Given a metric space $X$, an \emph{isometry} of $X$ is defined as a bijection that preserves the distance function. The set $\Isom(X)$ of all isometries of $X$ forms a group under composition, which is called the \emph{isometry group} of $X$.

An early result \cite{DantzigWaerden28} (see also \cite[pages 46--50]{KobayashiNomizu1}) states that if $X$ is connected and locally compact, then $\Isom(X)$, endowed with the compact-open topology, is also locally compact, and the isotropy group $\Isom(X)_x$ is compact for each $x\in X$. Moreover, if $X$ is itself compact, then so is $\Isom(X)$. Shortly thereafter, Myers and Steenrod proved that $\Isom(X)$ is a Lie transformation group when $X$ is a Riemannian manifold, see \cite{MS39}, \cite[page 41]{Kobayashi95}. Generalizing this, Fukaya and Yamaguchi have shown \cite{FY94} that $\Isom(X)$ is a Lie group whenever $X$ is an \emph{Alexandrov space}. Since then isometries of Alexandrov spaces have been actively studied, see for example \cite{GGG13, HS17}.

In the present article, we consider a special class of Alexandrov spaces. Namely, given a compact group $G$ acting linearly by isometries on a Euclidean vector space $V$, the \emph{orbit space} $X=V/G$ is an Alexandrov space with curvature bounded from below by $0$. A natural subgroup of $\Isom(V/G)$ consists of the isometries induced by  the \emph{$G$-equivariant} isometries of $V$, that is, isometries that commute with the $G$-action (see also Remark \ref{R:normalizer} below). Our main result is that, up to taking connected components, all isometries of $V/G$ arise in this way:
\begin{maintheorem}
\label{MT}
Let $V$ be an orthogonal representation space of the compact group $G$. Then, any element in the connected component $\Isom(V/G)_0$ of the Lie group $\Isom(V/G)$ lifts to a $G$-equivariant isometry of $V$.
\end{maintheorem}
The proof consists of showing that every isometry of $V/G$ that lifts to a diffeomorphism actually lifts to an isometry; and that every isometry in $\Isom(V/G)_0$ lifts to a diffeomorphism. The second statement follows from two results about the \emph{smooth structure} of $V/G$ (in the sense of G. Schwarz), namely \cite{AR17}, and  \cite[Corollary (2.4)]{Schwarz80}.

We note that elements of $\Isom(V/G)\setminus\Isom(V/G)_0$ do not necessarily lift, and that Theorem \ref{MT} does not generalize from Euclidean spaces to general Riemannian manifolds. See Section \ref{S:lifting} for details. 


Denote\todo{shortened this paragraph, and referred to the new Section 2} by $SV\subset V$ the unit sphere. Then  $\Isom(SV/G)$ can be identified with the isometries of $V/G$ fixing the origin, and forms a maximal compact subgroup of $\Isom(V/G)$, see Lemma \ref{L:trivial}(a) below.

Recall that every compact, connected Lie group is, up to a finite cover, a product of a torus and compact, simply-connected, simple Lie groups \cite[Theorem 8.1]{BroeckertomDieck}. The latter are classified, and fall into two types: classical ($\SU(n)$, $\operatorname{Spin}(n)$, and $\Sp(n)$), and exceptional ($G_2$, $F_4$, $E_6$, $E_7$, and $E_8$). As a corollary of Theorem \ref{MT}, we obtain the following structure result:
\begin{maintheorem}
\label{MT:cor}
Let $V$ be an orthogonal\todo{added ``orthogonal''} representation of a compact Lie group $G$. Then, up to a finite cover, $\Isom(SV/G)_0$ is isomorphic to a product of a torus and compact simple Lie groups of classical type. \end{maintheorem}

As another corollary of Theorem \ref{MT}, we obtain:
\begin{maintheorem}
\label{MT:cor2}
Let $V$ be an irreducible orthogonal representation of a compact Lie group $G$. Then  $\Isom(V/G)_0$ has rank at most one.
\end{maintheorem}
For a more detailed analysis of the case where $V$ is irreducible, see the end of Section \ref{S:computing}. 

As a final corollary of Theorem \ref{MT}, we give a simple formula for $\Isom(V/G)_0$ in the special, but in some sense typical, case where $V/G$ has no boundary, see Proposition \ref{P:boundary} and Remark \ref{R:boundary}.

This article is organized as follows. Section \ref{S:trivial} contains well-known preliminary facts about trivial factors.\todo{added sentence about new section 2} Section \ref{S:lifting} contains the proof of Theorem \ref{MT} and a few remarks. In Section \ref{S:computing} we extract some consequences of  Theorem \ref{MT} with the help of Schur's Lemma, and in particular prove Theorems \ref{MT:cor} and \ref{MT:cor2}, and we also describe the relationship with presence of boundary in the orbit space.


\subsection*{Acknowledgements}
I would like to thank C. Gorodski and A. Lytchak for discussions regarding the lifting of isometries that led to this project, and the anonymous referee\todo{added thanks to referee} for suggestions that improved the presentation.

\section{Trivial factors}
\label{S:trivial} \todo{New section}
This section contains preliminary observations regarding trivial factors, which are used multiple times in what follows. Let $G$ be a compact group acting linearly by isometries on the Euclidean vector space $V$. 

Denote by $SV\subset V$ the unit sphere in $V$, centered at the origin. Since $SV$ is preserved by the $G$-action, we have a well-defined quotient $SV/G$, and $\Isom(SV/G)$ can be identified with the subgroup of $\Isom(V/G)$ consisting of the isometries of $V/G$ which fix (the image of) the origin.

Denote by $F\subset V$ the subspace of vectors fixed by $G$, that is, the sum of the trivial factors of the $G$-representation $V$, and by $F^\perp$ its orthogonal complement in $V$. Thus $V$ can be identified with $F\times F^\perp$, where $G$ acts only on the second factor, and the quotient $V/G$ splits isometrically as the product $V/G=F\times (F^\perp/G)$. 

By  \cite[Section 5.1]{GL14}, the isometry group of $V/G$ also splits as
\[ \Isom(V/G)=\Isom(F)\times \Isom(F^\perp/G), \]
and, moreover, any isometry of $F^\perp/G$ must fix the origin. In other words, $\Isom(F^\perp/G)=\Isom(SF^\perp/G)$. For convenience, we gather a few consequences of these facts in the following:
\begin{lemma} \label{L:trivial}
In the notation above, we have: 
\begin{enumerate}[(a)]
\item The group $\Isom(V/G)$ has
\[ \Isom(SV/G)=\OO(F)\times \Isom(F^\perp/G)\]
as a maximal compact subgroup.
\item The connected component containing the identity splits as
\[ \Isom (V/G)_0=\Isom(F)_0\times \Isom(F^\perp/G)_0.\]
\item An isometry $(\phi, \psi)\in  \Isom(V/G)=\Isom(F)\times \Isom(F^\perp/G)$ lifts to a $G$-equivariant isometry of $V$ if and only if $\psi$ lifts to a $G$-equivariant isometry of $F^\perp$.
\end{enumerate}
\end{lemma}
These facts essentially mean that, for the purposes of the present article, we may assume that the $G$-representation $V$ has no trivial factors, that is, $F=0$.

\section{Lifting isometries}
\label{S:lifting}

The proof of Theorem \ref{MT} uses the smooth structure of orbit spaces, in the sense of G. Schwarz \cite{Schwarz80}. Given a $G$-manifold $M$, recall that a function $M/G\to\R$ is defined to be smooth if its composition with the projection $M\to M/G$ is a smooth function on $M$ in the usual sense; and that a map $M/G\to M'/G'$ is smooth if the pull back of smooth functions are smooth.

\begin{proof}[Proof of Theorem \ref{MT}]
Let $\varphi\in\Isom(V/G)_0$.
We may assume that $V$ has no trivial factors, so that $\varphi$ (as well as the isometries in a path connecting $\varphi$ to the identity) fix the origin --- see Section \ref{S:trivial},\todo{rephrased this sentence, referring to the new section 2} especially Lemma \ref{L:trivial}(b), (c).


By \cite{AR17}, the action of $\Isom(V/G)_0$ on $V/G$ is smooth in the sense of Schwarz, and so, by  \cite[Corollary (2.4)]{Schwarz80}, $\varphi$ lifts to a $G$-equivariant diffeomorphism $\tilde{\varphi}$ of $V$ that fixes the origin. To finish the proof it suffices to show that the \emph{differential}  $d\tilde{\varphi}_0$ at the origin is orthogonal and lifts $\varphi$.

Let $\lambda>0$, and denote by $\tilde{r}_\lambda:V\to V$ multiplication by $\lambda$, and by $r_\lambda:V/G\to V/G$ the induced map on $V/G$. Since $\varphi$ fixes the origin, it commutes with $r_\lambda$, so that $\tilde{r}_{\lambda}^{-1} \tilde{\varphi} \tilde{r}_\lambda$ is again a $G$-equivariant diffeomorphism of $V$ lifting $\varphi=r_\lambda^{-1}\varphi r_\lambda$. Taking $\lambda\to 0$, this converges to the differential $d\tilde{\varphi}_0$, which is then a linear $G$-equivariant lift of $\varphi$. This implies that $d\tilde{\varphi}_0$ takes the unit sphere $SV$ to itself, and in particular it is an orthogonal transformation.
\end{proof}

Next we mention ways in which Theorem \ref{MT} can and cannot be extended.

\begin{remark}
\label{R:connected}
The condition that $\varphi$ be in the \emph{connected component} of $\Isom(V/G)$ in the statement of Theorem \ref{MT} is necessary. To construct examples of non-liftable isometries, consider cohomogeneity two representations of $G$\todo{example has been phrased more generally}. Then the quotient $V/G$ is a sector in the plane, and its only non-trivial isometry is the reflection across the bisecting ray. Any lift of this isometry would swap the two non-trivial singular orbits corresponding to the two boundary rays of $V/G$. Thus, if these singular orbits are not diffeomorphic, such a lift cannot exist. For a concrete example, take $G=\SO(2)\times \SO(3)$ with its natural product action on $\R^2\times \R^3$. Then $V/G$ is a sector of angle $\pi/2$, and the two singular orbits are isometric to a $2$-sphere and a circle. For an example where the singular orbits have the same dimension, take the outer tensor product action of $G$ on $\R^2\otimes\R^3$. Here the quotient is isometric to a sector of angle $\pi/4$, and the isotropy groups along the two boundary rays are isomorphic to $\SO(2)$ and $\SO(2)\times\Z_2$ (see third line of Table E in \cite{GWZ08}). This means that the corresponding singular orbits are not diffeomorphic. More examples can be found in the same way by inspecting \cite[Table E]{GWZ08}.
 \end{remark}


\begin{remark}
\label{R:normalizer}
There is a natural class of isometries of $V$, larger than the group of $G$-equivariant isometries considered in Theorem \ref{MT}, which descend to isometries of $V/G$. It is given by the normalizer  $N(G)$ of the image of $G$ in $\OO(V)$. Remark \ref{R:connected} above shows that in general not every isometry of the quotient lifts to $N(G)$. 

However, when $G$ is finite, $N(G)$ does induce \emph{all} isometries $\varphi$ of $V/G$ that fix the origin, that is, all isometries of $SV/G$.  Indeed, the projection $\pi: SV\to SV/G$ and the map $\varphi\circ\pi: SV\to SV/G$ are universal orbifold coverings, hence $\varphi$ lifts to an isometry $f:SV\to SV$ (see \cite[Chapter 13]{Thurston} and \cite{Lange18}). Since the image of $G$ in $\OO(V)$ and its conjugate by $f$ are orbit-equivalent\footnote{For the purposes of this article, we say two subgroups $G_1, G_2\subset \OO(V)$, or more generally, two representations $G_1\to \OO(V)$, $G_2\to \OO(V)$, are \emph{orbit-equivalent} when they induce the same partition of $V$ into orbits.}\todo{clarified definition of ``orbit-equivalent''} finite subgroups of $\OO(V)$, they must coincide, that is, $f$ normalizes $G$.
\end{remark}

\begin{remark}
The conclusion of Theorem \ref{MT} may fail if the Euclidean space $V$ (or the round sphere $SV$) is replaced with a general Riemannian manifold. Indeed, consider the torus $M=S^1\times S^1$ with coordinates $x,y$, endowed with a warped product metric $dx^2+\rho(x)dy^2$. Then $G=S^1$ acts on $M$ by isometries (rotating the second $S^1$ factor\todo{added clarification}), with orbit space $M/G=S^1$. If the warping factor $\rho$ is not constant, that is, if the $G$-orbits do not all have the same length, then not every isometry in $\Isom(M/G)_0=S^1$ lifts to an isometry of $M$. 
\end{remark}

\section{Computing the connected component of the isometry group}
\label{S:computing}

The goal of this section is to give a general strategy to compute $\Isom(V/G)_0$, and to prove Theorems \ref{MT:cor} and \ref{MT:cor2}. We also give a simple formula for $\Isom(V/G)_0$ in the (in some sense typical) case where $V/G$ has no boundary, see Proposition \ref{P:boundary} and Remark \ref{R:boundary} below.

Let $V$ be an orthogonal representation of the compact group $G$, and assume for simplicity that $V$ is faithful (so that $G$ can be identified with a subgroup of $\OO(V)$), and that $V$ has no trivial factors (see Section \ref{S:trivial}\todo{added reference to Section 2}, especially Lemma \ref{L:trivial}(b)). Denote by $\Isom_G(V)$ the group of all $G$-equivariant isometries of $V$. Any element in this group descends to an isometry of the quotient, so we have a group homomorphism $\Isom_G(V)\to\Isom(V/G)$. Theorem \ref{MT} implies that the restriction to the identity component $\Isom_G(V)_0$ gives a surjective group homomorphism
\begin{equation} 
\label{E:p}
p: \Isom_G(V)_0 \to \Isom(V/G)_0.
\end{equation}
This points to a strategy for computing the isomorphism type of $\Isom(V/G)_0$, namely:\begin{enumerate}[(a)]
\item Determine the group $\Isom_G(V)_0$, and 
\item Determine the normal subgroup $\ker(p)$. 
\end{enumerate}

Task (a) is easily achieved using Schur's Lemma (see, for example, \cite[page 69]{BroeckertomDieck}). More precisely, one first decomposes 
\[V\cong \bigoplus_{i=1}^k V_i^{\oplus n_i}\]
where $V_i$ are pair-wise non-isomorphic irreducible $G$-representations. Each $V_i$ is of real, complex, or quaternionic type, according to whether the skew field $\operatorname{Hom}_G(V_i, V_i)$ is isomorphic to $\R$, $\C$, or $\HH$ (see  \cite[II.6]{BroeckertomDieck}). One then has the corresponding decomposition
\[\Isom_G(V)_0=\prod_{i=1}^k H_i\]
with each $H_i$ isomorphic to $\SO(n_i)$, $\U(n_i)$, or $\Sp(n_i)$, according to the type of $V_i$.

\begin{proof}[Proof of Theorem \ref{MT:cor2}]
When $V$ is irreducible, we have $k=1$ and $n_1=1$, so the ranks of $\Isom_G(V)_0$ and of $\Isom(V/G)_0$ are at most one. \todo{added proof of Theorem C}
\end{proof}
For a more detailed analysis of the irreducible case, see the end of Section \ref{S:computing}.

\begin{proof}[Proof of Theorem \ref{MT:cor}]
By Lemma \ref{L:trivial}(a),\todo{assume from the start of the proof that there are no trivial factors} we may assume $V$ has no trivial factors. By the discussion above, the Lie algebra of $\Isom(SV/G)=\Isom(V/G)$ is the quotient of a Lie algebra $\frak{h}$ by an ideal $I$, where $\frak{h}$ is the direct sum of a Abelian factor $\frak{a}$ with a finite number of simple Lie algebras isomorphic to  $\so(n)$, $\su(n)$, or $\spp(n)$. Then $I$ must be equal to the direct sum of some subspace of $\frak{a}$ with a finite collection of the simple factors of $\frak{h}$, by uniqueness of simple ideals. Thus the quotient $\frak{h}/I$ has the desired form.
\end{proof}

Task (b) can be approached using the characterization of $\ker(p)$ as the largest subgroup $L$ of $\Isom_G(V)_0$ with the property that the natural action of $L\times G$ on $V$ is orbit-equivalent to the original $G$-action. For example, when $G$ is finite, this implies that $\ker(p)$ is also finite. This characterization also leads to the following relationship with the presence of boundary (that is, codimension-one strata) in the orbit space:
\begin{proposition}
\label{P:boundary}
Let $G$ be a compact group, $V$  an orthogonal $G$-representation without trivial factors, and $p: \Isom_G(V)_0 \to \Isom(V/G)_0$ the surjective homomorphism considered in \eqref{E:p} above. Denote by $Z(G)$ the center of $G$. Then
\begin{enumerate}[(a)]
\item $\ker(p)$ contains $Z(G)\cap \Isom_G(V)_0$.
\item If $V/G$ has no boundary, then $\ker(p)=Z(G)\cap \Isom_G(V)_0$, so that
\[ \Isom(V/G)_0= \frac{\Isom_G(V)_0}{Z(G)\cap \Isom_G(V)_0}.\]
\end{enumerate}
\end{proposition}
\begin{proof}
\begin{enumerate}[(a)]
\item Elements of $Z(G)$ are clearly $G$-equivariant isometries, and act trivially on the quotient.
\item Assume for a contradiction that $\ker(p)$ is strictly larger than $Z(G)\cap \Isom_G(V)_0$. Then $G$ and $\hat{G}:=G\times \ker(p)$ have different images in $\OO(V)$, because $Z(G)\cap \Isom_G(V)_0=G\cap \Isom_G(V)_0$. 
Let $x\in SV$ be a principal point for the actions of both $G$ and $\hat{G}$. Since these actions are orbit-equivalent, we have in particular $G\cdot x={\hat{G}\cdot x}$. These are naturally identified with $G/G_x$ and $\hat{G}/\hat{G}_x$, which, together with the fact that $G$ and $\hat{G}$ have distinct images in $\OO(V)$, implies that the principal isotropy group $\hat{G}_x$ acts non-trivially on $V$. In other words, the fixed-point set $V^{\hat{G}_x}$ is a proper subspace of $V$, so that we have a non-trivial Luna--Richardson type reduction \cite{LR79}
\[ \frac{V}{G} = \frac{V}{\hat{G}} = \frac{V^{\hat{G}_x}}{N_{\hat{G}}(\hat{G}_x)} .\]
By \cite[Proposition 1.1]{GL14}, the orbit space $V/G$ has boundary, contradiction.
\end{enumerate}
\end{proof}

\begin{example}
Consider the Hopf action of $G=\U(1)$ on $V=\R^4=\C^2$. As a representation, $V$ decomposes as the direct sum of two copies of the irreducible representation, of complex type, of $G$ on $\R^2=\C$. By Schur's Lemma, $\Isom_G(V)_0$ is isomorphic to $\U(2)$. The quotient $SV/G$ is isometric to the $2$-sphere of radius $1/2$, and hence $\Isom(V/G)_0=\Isom(SV/G)_0$ is isomorphic to $\SO(3)$. Thus the map $p$ in Equation \eqref{E:p} corresponds to the quotient of $\U(2)$ by its center, which is also equal to (the center of) $G$:
\[ p: \U(2)\to \PSU(2)\cong\SO(3).\]
Since $V/G$ has no boundary, this illustrates Proposition \ref{P:boundary}(b).
\end{example}

\begin{remark}
\label{R:boundary}
In some sense lack of boundary in $V/G$ is the typical case, and thus Proposition \ref{P:boundary}(b)  allows one to easily compute $\Isom(V/G)_0$ for ``most'' representations. For example, if $G$ is a simple Lie group, the classification of all representations $V$ such that $V/G$ has boundary has been announced by\todo{removed citation to preprint GKW, and use the phrase ``announced by''}
 C. Gorodski, A. Kollross, and B. Wilking. In particular, one would be able to decide whether Proposition \ref{P:boundary}(b) is applicable by checking a given representation against their list.
\end{remark}

We conclude by specializing the discussion in this section to the case where $V$ is irreducible, thus sharpening the conclusion of Theorem \ref{MT:cor2}.\todo{replaced ``give a proof of Thm C'' with ``sharpening the conclusion of Thm C''}

If $V$ is of real type, then $\Isom(V/G)_0$ is a quotient of $\SO(1)$, hence trivial. That is, $\Isom(V/G)$ is a finite group.

If $V$ is of complex type, then $\Isom(V/G)_0$ is a quotient of $\U(1)$, hence either trivial, or isomorphic to $\U(1)$. Both cases occur. For example, if $G$ is finite, then $\Isom(V/G)_0\cong \U(1)$, because $G$ and $\U(1)\times G$ cannot be orbit-equivalent. On the other hand, the representation of $\U(1)\times G$ on $V$ is still irreducible of complex type, and $\Isom(V/(\U(1)\times G))_0$ is trivial.

If $V$ is of quaternionic type, then $\Isom(V/G)_0$ is  a quotient of $\Sp(1)$, hence either trivial, or isomorphic to either $\Sp(1)$ or $\SO(3)$.\todo{removed application of preprint ``GKW''}


\end{document}